\documentclass[11pt,a4paper]{article}

\usepackage[T1]{fontenc}
\usepackage[utf8]{inputenc}
\usepackage[a4paper,margin=1in]{geometry}
\usepackage{amsmath,amssymb,amsfonts,amsthm,mathrsfs}
\usepackage{graphicx}
\usepackage{booktabs}
\usepackage{subfigure}
\usepackage{authblk}
\usepackage{microtype}
\usepackage[hidelinks]{hyperref}
\hypersetup{
  pdftitle={A mesh- and horizon-robust preconditioner for finite element discretizations of volume-constrained nonlocal diffusion problems},
  pdfauthor={Jiashu Lu, Yufeng Nie, Haolun Zhang, and Pingrui Zhang},
  pdfkeywords={nonlocal diffusion, pair-cell assembly, asymptotically robust preconditioner}
}

\numberwithin{equation}{section}
\allowdisplaybreaks
\theoremstyle{plain}
\newtheorem{theorem}{Theorem}[section]
\newtheorem{lemma}[theorem]{Lemma}
\newtheorem{proposition}[theorem]{Proposition}
\theoremstyle{definition}
\newtheorem{assumption}[theorem]{Assumption}
\theoremstyle{remark}
\newtheorem*{remark}{Remark}

\begin{document}

\title{A mesh- and horizon-robust preconditioner for finite element discretizations of volume-constrained nonlocal diffusion problems}
\author[1]{Jiashu Lu}
\author[2]{Yufeng Nie}
\author[2]{Haolun Zhang}
\author[3]{Pingrui Zhang}
\affil[1]{School of Mathematics, Xi'an University of Technology, Xi'an 710054, China}
\affil[2]{School of Mathematics and Statistics, Northwestern Polytechnical University, Xi'an 710129, China}
\affil[3]{Shaoxing Institute, Zhejiang University, Shaoxing 312099, China}
\date{}

\maketitle

\begin{abstract}
In this paper, we develop a preconditioner that is robust with respect to both the mesh size and the interaction horizon for finite element discretizations of volume-constrained nonlocal diffusion problems. Motivated by the Fourier multiplier of the nonlocal operator, we construct the preconditioner $B_h=G_h+\beta\delta^2D_h^{-1}$, where $G_h$ is a symmetric local Poisson solver, $D_h$ is the lumped mass matrix, and $\beta$ is a kernel-dependent constant. We rigorously prove that the spectrum of the $B_h$-preconditioned system is bounded above and away from zero uniformly with respect to the mesh size and horizon. Moreover, we develop a pair-cell assembly for the finite element matrix that preserves the symmetry and positive definiteness required by the preconditioner and PCG, and establish an $L^2$-error estimate for the finite element discretization. Finally, we use a single symmetric geometric multigrid V-cycle for $G_h$ in the numerical experiments, and the corresponding numerical results confirm the predicted convergence rate and show bounded condition numbers and PCG iteration counts over the tested meshes, horizons, and kernels.
\end{abstract}



\section{Introduction}
Nonlocal models have attracted considerable attention in computational mathematics and computational mechanics because of their ability to describe long-range interactions and discontinuities. Unlike local partial differential equation models, nonlocal models describe spatial interactions through integral operators and are therefore naturally suited to modeling and analyzing problems involving discontinuities and anomalous diffusion. For example, the peridynamic model provides a nonlocal counterpart of continuum mechanics and has been used to describe cracks and their evolution\cite{silling2000reformulation}, including crack initiation and propagation\cite{dipasquale2014crack,panchadhara2016application,zhou2025research}. Moreover, volume-constrained nonlocal models can describe general Markov jump processes in bounded domains\cite{du2012analysis,du2014nonlocal} and provide a natural connection between classical and fractional diffusion\cite{d2013fractional,tian2016asymptotically,du2023nonlocal,du2025error}. 

Although nonlocal models offer greater modeling flexibility than classical partial differential equation models, the nonlocality also creates substantial difficulties in numerical computation. These difficulties arise mainly in assembling the discrete matrices and solving the resulting linear systems. Up to now, several approaches have been developed to address these two difficulties. For matrix assembly, in the light of the asymptotically compatible framework\cite{tian2014asymptotically}, efficient numerical methods and multi-dimensional implementations are developed\cite{d2021cookbook,chen2024efficient,lu2026nonlocal}. For the solution of the discrete systems, the fast Fourier transform (FFT) type method and the  multigrid methods have been developed for structured nonlocal discretizations\cite{chen2017convergence,chen2024fast,tian2024fast,liu2025fast}, and Schwarz methods have been studied for nonlocal Dirichlet problems\cite{schuster2025schwarz}. More recently, uniform spectral bounds have been established for a $\tau$-preconditioner for a class of two-level Toeplitz systems arising from finite difference discretizations of nonlocal diffusion\cite{zhang2026spectral}. These studies have made substantial progress in the numerical treatment of nonlocal models. However, for large-scale conforming finite element discretizations on quasi-uniform meshes, the coefficient matrices can be severely ill-conditioned, with condition numbers that depend strongly on $h$ and $\delta$. An effective preconditioner robust with respect to both \(h\) and \(\delta\) is therefore still needed.

To address this gap, we develop an additive preconditioner that is robust with respect to both $h$ and $\delta$, and we call this property \emph{asymptotic robustness} for a preconditioner. To obtain this robustness, we first show that the Fourier symbol of the nonlocal operator on $\mathbb R^d$ is uniformly comparable to $|\xi|^2/(1+\delta^2|\xi|^2)$ under the kernel assumptions in Section \ref{sec2}. Therefore, for $\xi\ne0$, the symbol of the inverse operator is uniformly comparable to $|\xi|^{-2}+\delta^2$. After finite element discretization, these two terms correspond to a local Poisson solver and an inverse mass term, respectively. This leads to the preconditioner:
\begin{equation*}
	B_h=G_h+\beta\delta^2D_h^{-1}
\end{equation*}
Here $G_h$ is a symmetric local Poisson solver, $D_h$ is the lumped mass matrix, and $\beta$ is a constant determined by the kernel mass. 

The main contributions of this paper are summarized as follows:
\begin{itemize}
	\item[(i)] We rigorously prove that there exist constants $\delta_0,c,C>0$ such that for all $0<\delta\le\delta_0$,
	\begin{equation*}
		\sigma\!\left(B_h^{1/2}A_hB_h^{1/2}\right)\subset[c,C].
	\end{equation*}
	The spectral bounds are independent of $h$, $\delta$, and kernel partitions. The result holds for any symmetric local solver $G_h$ uniformly spectrally equivalent to the inverse local stiffness matrix. Consequently, the PCG convergence rate is uniform in these parameters.
	\item[(ii)] We develop a pair-cell assembly that exactly realizes the cell-averaged Galerkin matrix and preserves its symmetry and positive definiteness.
	\item[(iii)] We derive an $L^2$-error estimate for the finite element approximation with the proposed pair-cell assembly. The $O(h^2)$ convergence rate is obtained under the stated regularity and consistency assumptions. And the numerical convergence test further confirms the theoretical result.
	\item[(iv)] We realize $G_h$ by a single symmetric geometric multigrid V-cycle and test the resulting preconditioner over a large range of $h$, $\delta$, and kernels. The condition numbers and PCG iteration counts remain bounded, confirming the robustness and effectiveness of $B_h$.
\end{itemize}
The remainder of the paper is organized as follows. Section \ref{sec2} presents the nonlocal problem, finite element discretization, pair-cell assembly, and preconditioner. Sections \ref{sec3} and \ref{sec4} establish the continuous and discrete estimates, respectively. Section \ref{sec5} proves the main theorem. Section \ref{sec6} reports the numerical results, and Section \ref{sec7} concludes the paper.

\section{Nonlocal finite element formulation and asymptotically robust preconditioner}\label{sec2}
\subsection{Nonlocal diffusion problem}
Let $d\ge1$ denote the spatial dimension, let $\Omega\subset\mathbb R^d$ be a bounded domain, and let $B_r(x)$ denote the open ball centered at $x$ with radius $r$. Set $B_r=B_r(0)$ and $\omega_d=|B_1|$. For a horizon parameter $\delta>0$, define the interaction domain by
\begin{equation*}
\Omega_I^\delta
=\left\{y\in\mathbb R^d\setminus\Omega:
y\in B_\delta(x)\text{ for some }x\in\Omega\right\},
\end{equation*}
 and we set $\widehat\Omega_\delta=\Omega\cup\Omega_I^\delta$. Given $f\in L^2(\Omega)$ and $g_\delta\in L^2(\Omega_I^\delta)$, the Dirichlet volume-constrained nonlocal diffusion problem for $u_\delta\in L^2(\widehat\Omega_\delta)$ is
\begin{equation}\label{modelequation}
	\begin{cases}
		-\mathscr L_\delta u_\delta=f & \text{in }\Omega,\\
		u_\delta=g_\delta & \text{in }\Omega_I^\delta.
	\end{cases}
\end{equation}
where $\mathscr L_\delta$ is the nonlocal diffusion operator defined by
\begin{equation*}
	(\mathscr L_\delta u)(x)
	=2\int_{B_\delta(x)}\gamma_\delta(x,y)
	\bigl(u(y)-u(x)\bigr)\,dy,
	\qquad x\in\Omega,
\end{equation*}
and the kernel function $\gamma_\delta(x,y)$  is given by
\begin{equation*}
	\gamma_\delta(x,y)
	=\delta^{-(d+2)}\rho\!\left(\frac{y-x}{\delta}\right).
\end{equation*}
Here $\rho$ satisfies
\begin{equation}\label{2.1}
	\rho\in L^1(B_1),\quad
	\rho(z)=\rho(-z),\quad
	\rho(z)\ge0\quad\text{a.e. in }B_1.
\end{equation}
Moreover, we assume that there exist constants $M_\rho,\rho_*>0$ and $0<r_*\le1$ such that
\begin{equation}\label{2.2}
	\|\rho\|_{L^1(B_1)}\le M_\rho,\qquad
	\rho(z)\ge\rho_*
	\quad\text{for a.e. }z\in B_{r_*}.
\end{equation}

Next, we define the kernel mass and the scaling parameter of the preconditioner by
\begin{equation}\label{2.3}
	\mu_\rho=\int_{B_1}\rho(z)\,dz,\qquad
	\beta=\frac{1}{2\mu_\rho}.
\end{equation}
Then \eqref{2.2} yields
\begin{equation}\label{2.4}
	\beta_-:=\frac{1}{2M_\rho}
	\le\beta\le
	\frac{1}{2\rho_*\omega_dr_*^d}
	=:\beta_+.
\end{equation}

Define $\ell_\delta\in L^2(\widehat\Omega_\delta)$ by $\ell_\delta=0$ in $\Omega$ and $\ell_\delta=g_\delta$ in
$\Omega_I^\delta$, and set $w_\delta=u_\delta-\ell_\delta$ on $\widehat\Omega_\delta$. Then $w_\delta=0$ in $\Omega_I^\delta$ and $u_\delta=w_\delta+\ell_\delta$. For $v\in L^2(\Omega)$, let $E_0v$ denote its zero extension to $\mathbb R^d$. Identifying $w_\delta$ with its restriction to $\Omega$, the model equation \eqref{modelequation} can be rewritten as
\begin{equation}\label{lifted}
-\mathscr L_\delta(E_0w_\delta)=F_\delta\quad\text{in }\Omega,
\qquad
F_\delta(x)=f(x)
+2\int_{\Omega_I^\delta\cap B_\delta(x)}
\gamma_\delta(x,y)g_\delta(y)\,dy.
\end{equation}
After extending $g_\delta$ by zero outside $\Omega_I^\delta$, the integral in $F_\delta(x)$ is a convolution of an $L^1$ kernel with an $L^2$ function. Young's convolution inequality therefore shows that $F_\delta\in L^2(\Omega)$.

Moreover, for $u,v\in L^2(\Omega)$, we define the symmetric bilinear form
\begin{equation}\label{2.5}
	\begin{aligned}
		a_\delta(u,v)
		&=\int_{\mathbb R^d}\int_{B_\delta(x)}
		\delta^{-(d+2)}\rho\!\left(\frac{y-x}{\delta}\right)
		\bigl(E_0u(y)-E_0u(x)\bigr)
		\bigl(E_0v(y)-E_0v(x)\bigr)\,dy\,dx .
	\end{aligned}
\end{equation}
The evenness of $\rho$ gives $(-\mathscr L_\delta(E_0u),v)_{L^2(\Omega)}=a_\delta(u,v)$. Hence the variational formulation of \eqref{lifted} is to find $w_\delta\in L^2(\Omega)$ such that
\begin{equation}\label{liftproblem}
	a_\delta(w_\delta,v)=(F_\delta,v)_{L^2(\Omega)}
	\quad\forall v\in L^2(\Omega).
\end{equation}
\subsection{Finite element discretization and pair-cell matrix assembly}\label{sec2.2}
Assume in this subsection that $\Omega$ is a bounded Lipschitz polyhedral domain. To approximate the nonlocal integrals numerically, we replace the spherical interaction neighborhood $B_\delta(x)$ by $x+\delta P$, where $P\subset\mathbb R^d$ is a centrally symmetric $d$-dimensional polytope approximating $B_1$ and satisfying
\begin{equation}\label{2.6}
	P=-P,\qquad B_{1/2}\subset P\subset\overline{B_1}.
\end{equation}
The corresponding polytopal interaction layer and extended domain are defined as
\begin{equation*}
	\Omega_{I,P}^\delta
	=\left\{y\notin\Omega:\ y\in x+\delta P
	\text{ for some }x\in\Omega\right\},
	\qquad
	\widehat\Omega_{\delta,P}=\Omega\cup\Omega_{I,P}^\delta.
\end{equation*}

To approximate $\rho$ on $P$, let $\mathscr Z_q$ be a finite partition of $P$ into polytopal cells $Z$ of positive measure. The cells cover $P$ without positive-measure overlap, and $Z\in\mathscr Z_q$ implies $-Z\in\mathscr Z_q$. Then we define the maximum diameter of $\mathscr Z_q$ by
\begin{equation}\label{2.7}
	q:=\max_{Z\in\mathscr Z_q}\operatorname{diam}Z,
	\quad 0<q\le r_*/2.
\end{equation}
For each kernel cell $Z$, we define
\begin{equation}\label{2.8}
	\mu_Z=\int_{Z\cap B_1}\rho(z)\,dz,\quad
	\rho^{\rm av}|_Z=\frac{\mu_Z}{|Z|}.
\end{equation}
Thus $\rho^{\rm av}$ is constant on each kernel cell and satisfies $\int_Z\rho^{\rm av}=\int_{Z\cap B_1}\rho$ for every $Z\in\mathscr Z_q$. We then extend $\rho^{\rm av}$ by zero outside $P$. Use $\rho^{\rm av}$ in place of $\rho$ in \eqref{2.5}, and denote the resulting bilinear form by $a_{\delta,\rho^{\rm av}}$. Since $\rho^{\rm av}$ is supported in $P$ and the functions are extended by zero outside $\Omega$, the form involves only ordered pairs $(x,y)$ in the interaction domain defined by
\begin{equation*}
	\mathcal I_{\delta,P}
	=\left\{(x,y)\in(\widehat\Omega_{\delta,P})^2:
	\begin{array}{l}
		x\in\Omega\text{ or }y\in\Omega,\\
		y-x\in\delta P
	\end{array}\right\}.
\end{equation*}
Equivalently, for $u,v\in L^2(\Omega)$,
\begin{equation*}
	\begin{aligned}
		a_{\delta,\rho^{\rm av}}(u,v)
		&=\delta^{-(d+2)}
		\int_{\mathcal I_{\delta,P}}
		\rho^{\rm av}\!\left(\frac{y-x}{\delta}\right)
		\bigl(E_0u(y)-E_0u(x)\bigr)\\
		&\hspace{34mm}\times
		\bigl(E_0v(y)-E_0v(x)\bigr)\,dy\,dx.
	\end{aligned}
\end{equation*}

Let $\mathcal T_h$ be a conforming simplicial mesh of $\Omega$ fitted to $\partial\Omega$, and set $h=\max_{T\in\mathcal T_h}\operatorname{diam}T$. Assume that these meshes are uniformly shape regular and globally quasi-uniform as $h$ varies. For each $T\in\mathcal T_h$, let $\mathbb P_1(T)$ denote the space of affine polynomials on $T$, and define
\begin{equation}\label{2.9}
	V_h^0=\left\{
	v_h\in C^0(\overline\Omega):
	v_h|_T\in\mathbb P_1(T)\ \forall T\in\mathcal T_h,\quad
	v_h|_{\partial\Omega}=0
	\right\}.
\end{equation}
Then we let $N_h=\dim V_h^0$, and let $\{\phi_i\}_{i=1}^{N_h}$ be the free-node nodal basis of the space in \eqref{2.9}.

Extend $\mathcal T_h$ to a conforming simplicial mesh $\widehat{\mathcal T}_h$ of $\widehat\Omega_{\delta,P}$, with $\partial\Omega$ as an element interface, and call the elements contained in $\Omega_{I,P}^\delta$ the exterior mesh. Let $g_{\delta,h}$ be a prescribed elementwise affine approximation of $g_\delta|_{\Omega_{I,P}^\delta}$ on the exterior mesh, and define the discrete lifting $\ell_{\delta,h}$ by
\begin{equation*}
	\ell_{\delta,h}(x)=
	\begin{cases}
		0,&x\in\Omega,\\
		g_{\delta,h}(x),&x\in\Omega_{I,P}^\delta.
	\end{cases}
\end{equation*}
We further define the effective load $F_h\in L^2(\Omega)$ by
\begin{equation*}
	F_h(x)=f(x)
	+2\delta^{-(d+2)}
	\int_{\Omega_{I,P}^\delta\cap(x+\delta P)}
	\rho^{\rm av}\!\left(\frac{y-x}{\delta}\right)
	g_{\delta,h}(y)\,dy,
	\qquad x\in\Omega.
\end{equation*}
Then the Galerkin approximation is to find $w_h\in V_h^0$ such that
\begin{equation}\label{galerkin}
	a_{\delta,\rho^{\rm av}}(w_h,v_h)
	=(F_h,v_h)_{L^2(\Omega)}
	\qquad\forall v_h\in V_h^0.
\end{equation}
By writing
\begin{equation*}
	w_h=\sum_{j=1}^{N_h}U_j\phi_j\in V_h^0,\qquad
	U=(U_1,\ldots,U_{N_h})^T,
	\qquad
	u_{\delta,h}=\ell_{\delta,h}+E_0w_h
	\quad\text{on }\widehat\Omega_{\delta,P},
\end{equation*}
the finite element approximation can be written in matrix form as
\begin{equation*}
	A_hU=b_h,
\end{equation*}
where
\begin{equation*}
	\begin{aligned}
		(A_h)_{ij}
		&=a_{\delta,\rho^{\rm av}}(\phi_j,\phi_i)\\
		&=\delta^{-(d+2)}
		\int_{\mathcal I_{\delta,P}}
		\rho^{\rm av}\!\left(\frac{y-x}{\delta}\right)
		\bigl(E_0\phi_j(y)-E_0\phi_j(x)\bigr)\\
		&\hspace{36mm}\times
		\bigl(E_0\phi_i(y)-E_0\phi_i(x)\bigr)\,dy\,dx,
		\qquad 1\le i,j\le N_h,
	\end{aligned}
\end{equation*}
and
\begin{equation*}
	\begin{aligned}
		(b_h)_i
		&=(F_h,\phi_i)_{L^2(\Omega)}\\
		&=(f,\phi_i)_{L^2(\Omega)}\\
		&\quad-\delta^{-(d+2)}
		\int_{\mathcal I_{\delta,P}}
		\rho^{\rm av}\!\left(\frac{y-x}{\delta}\right)
		\bigl(\ell_{\delta,h}(y)-\ell_{\delta,h}(x)\bigr)\\
		&\hspace{37mm}\times
		\bigl(E_0\phi_i(y)-E_0\phi_i(x)\bigr)\,dy\,dx,
		\qquad 1\le i\le N_h.
	\end{aligned}
\end{equation*}
Proposition \ref{proposition4.1} shows that $A_h$ is symmetric positive definite. We next construct a pair-cell assembly that preserves this property exactly under quadrature.

To evaluate these nonlocal double integrals, element-pair decompositions combined with nested inner--outer quadrature have been considered\cite{d2021cookbook}. Our construction also starts from an element pair $(T,S)$, but treats $(x,y)$ jointly in $\mathbb R^{2d}$ instead of integrating with respect to $x$ and $y$ successively. To preserve symmetry and positive semidefiniteness under quadrature, we further divide each element-pair interaction region according to the kernel cell containing $z=(y-x)/\delta$. For $T,S\in\widehat{\mathcal T}_h$ and $Z\in\mathscr Z_q$, define the pair cell
\begin{equation}\label{2.10}
	\mathcal C_{TSZ}=
	\left\{(x,y)\in(T\times S)\cap\mathcal I_{\delta,P}:
	(y-x)/\delta\in Z\right\}.
\end{equation}
Since $\widehat{\mathcal T}_h$ partitions $\widehat\Omega_{\delta,P}$ and $\mathscr Z_q$ partitions $P$, the pair cells in \eqref{2.10} cover $\mathcal I_{\delta,P}$ without positive-measure overlap. On $\mathcal C_{TSZ}$, $(y-x)/\delta\in Z$, and hence $\rho^{\rm av}((y-x)/\delta)=\mu_Z/|Z|$ is constant. On $T\times S$, the terms $E_0\phi_i(y)-E_0\phi_i(x)$ and $\ell_{\delta,h}(y)-\ell_{\delta,h}(x)$ are affine in $(x,y)$. Thus the integrands in $A_h$ and the lifting term in $b_h$ are polynomials of total degree at most two on each pair cell. We therefore
partition each pair cell into nonoverlapping $2d$-simplices and apply the following positive quadrature rule on each simplex.

For a $2d$-simplex $R=\operatorname{conv}\{v_0,\ldots,v_{2d}\}$, let $|R|$ denote its $2d$-dimensional measure. The quadrature nodes and weights are
\begin{equation}\label{2.11}
	\zeta_i=\theta_d v_i+\vartheta_d\sum_{j\ne i}v_j,\qquad
	\omega_i=\frac{|R|}{2d+1},\qquad i=0,\ldots,2d,
\end{equation}
where
\begin{equation}\label{2.12}
	\theta_d=\frac{1+2d/\sqrt{2d+2}}{2d+1},\qquad
	\vartheta_d=\frac{1-1/\sqrt{2d+2}}{2d+1}.
\end{equation}
The identities $\theta_d+2d\vartheta_d=1$, $\theta_d^2+2d\vartheta_d^2=1/(d+1)$, and $2\theta_d\vartheta_d+(2d-1)\vartheta_d^2=1/(2d+2)$ show that \eqref{2.11} is exact for polynomials of total degree at most two. Applying this quadrature rule to every simplex therefore evaluates the nonlocal integrals in \(A_h\) and \(b_h\) exactly. Denote the coordinates and weight of node \(\nu\) by \((x_\nu,y_\nu)\) and \(\omega_\nu\), and let \(Z(\nu)\) denote the kernel cell associated with its parent pair cell. Set
\begin{equation}\label{2.13}
	\begin{aligned}
		(d_\nu)_i&=E_0\phi_i(y_\nu)-E_0\phi_i(x_\nu),&
		e_\nu&=\ell_{\delta,h}(y_\nu)-\ell_{\delta,h}(x_\nu),\\
		c_\nu&=\delta^{-(d+2)}\bigl(\rho^{\rm av}|_{Z(\nu)}\bigr)\omega_\nu,
	\end{aligned}
\end{equation}
then summing over all quadrature nodes, we have
\begin{equation}\label{2.14}
	A_h=\sum_\nu c_\nu d_\nu d_\nu^T.
\end{equation}
And the load vector is
\begin{equation*}
	(b_h)_i=(f,\phi_i)_{L^2(\Omega)}
	-\sum_\nu c_\nu e_\nu(d_\nu)_i,
	\qquad 1\le i\le N_h.
\end{equation*}
\subsection{Asymptotically robust preconditioner}
Let $K_h$ and $M_h$ denote the stiffness and consistent mass matrices on $V_h^0$:
\begin{equation}\label{2.15}
	(K_h)_{ij}=\int_\Omega\nabla\phi_j\cdot\nabla\phi_i\,dx,\quad
	(M_h)_{ij}=\int_\Omega\phi_j\phi_i\,dx.
\end{equation}
Both matrices are symmetric positive definite. We also define the diagonal lumped mass matrix $D_h$ used in the preconditioner. Let
$\{\phi_j^{\rm full}\}_{j=1}^{N_h^{\rm full}}$ be the full nodal basis on
$\mathcal T_h$, including the basis functions associated with the nodes on
$\partial\Omega$, and let $M_h^{\rm full}$ be the corresponding mass matrix.
We first apply row-sum lumping to $M_h^{\rm full}$ and then remove the rows
and columns associated with the boundary nodes. The remaining
$N_h\times N_h$ diagonal matrix is $D_h$. Since the full nodal basis forms
a partition of unity, its diagonal entries satisfy
\begin{equation}\label{2.16}
	(D_h)_{ii}
	=
	\sum_{j=1}^{N_h^{\rm full}}
	\int_\Omega
	\phi_i(x)\phi_j^{\rm full}(x)\,dx
	=
	\int_\Omega\phi_i(x)\,dx,
	\qquad 1\le i\le N_h.
\end{equation}

Moreover, to approximate $K_h^{-1}$, we let $G_h$ be a symmetric local Poisson solver that satisfies the following assumption.
\begin{assumption}[Local solver]\label{assumption2.1}
	The local Poisson solver $G_h$ is symmetric and satisfies
	\begin{equation}\label{2.17}
		c_GX^TK_h^{-1}X
		\le X^TG_hX
		\le C_GX^TK_h^{-1}X
		\quad\forall X\in\mathbb R^{N_h},
	\end{equation}
	where $c_G,C_G>0$ are independent of $h$, $\delta$, $\rho$, $P$, and $q$.
\end{assumption}

For the uniformly refined mesh hierarchy and the symmetric weighted-Jacobi V-cycle used in Section \ref{sec6}, the multigrid convergence results in \cite{braess1983new,brannick2008uniform} imply that Assumption \ref{assumption2.1} holds.

Then, we define the following preconditioner for the nonlocal diffusion model:
\begin{equation}\label{2.18}
	B_h=G_h+\beta\delta^2D_h^{-1}.
\end{equation}
The next theorem shows that $B_h$ is asymptotically robust: for all sufficiently small $\delta$, the spectrum of $B_h^{1/2}A_hB_h^{1/2}$ is bounded above and away from zero uniformly in $\delta$ and the discretization parameters.
\begin{theorem}[asymptotic robustness]\label{theorem2.2}
	Assume that  \eqref{2.1}, \eqref{2.2}, the geometric and discretization assumptions of Section \ref{sec2.2}, and Assumption \ref{assumption2.1} hold. Then there exist constants $\delta_0,c,C>0$ such that, for every $0<\delta\le\delta_0$,
	\begin{equation}\label{2.19}
		\sigma\!\left(
		B_h^{1/2}
		A_h
		B_h^{1/2}
		\right)
		\subset[c,C],
	\end{equation}
	where $\sigma$ denotes the spectrum. Moreover, with $\kappa(X)$ denoting the spectral condition number of a symmetric positive definite matrix $X$, we have
	\begin{equation}\label{2.20}
		\kappa\!\left(
		B_h^{1/2}
		A_h
		B_h^{1/2}
		\right)
		\le\frac{C}{c}.
	\end{equation}
	The above constants depend only on $d$, $\Omega$, the kernel-class bounds in \eqref{2.2}, the uniform mesh constants, and $c_G,C_G$. In particular, they are independent of $h$, $\delta$, $\delta/h$, $P$, and $q$. 
\end{theorem}

The proof of Theorem \ref{theorem2.2} is given in Section \ref{sec5}. Consequently, the PCG convergence rate is uniform in $h$, $\delta$, and $\delta/h$; see \cite[Chap. 9]{saad2003iterative}.

\section{Continuous estimates}\label{sec3}
We first establish uniform Fourier bounds and then transfer them to $\Omega$ by zero extension. This gives the continuous energy equivalence used in Sections \ref{sec4} and \ref{sec5}. Throughout Sections \ref{sec3}--\ref{sec5}, we set $t=\beta\delta^2$.

We use the Fourier transform $\widehat\varphi(\xi)=(2\pi)^{-d/2}\int_{\mathbb R^d}e^{-ix\cdot\xi}\varphi(x)\,dx$. And for an even, nonnegative general kernel $k\in L^1(B_1)$ and $\eta,\xi\in\mathbb R^d$, define
\begin{equation}\label{3.1}
	\Psi_k(\eta)
	=2\int_{B_1}k(z)\bigl(1-\cos(\eta\cdot z)\bigr)\,dz,
	\quad
	s_{\delta,k}(\xi)=\delta^{-2}\Psi_k(\delta\xi).
\end{equation}
Taking $k=\rho$ in \eqref{3.1}, the Fourier transform of the original nonlocal operator satisfies
$\widehat{-\mathscr L_\delta\varphi}(\xi) =s_{\delta,\rho}(\xi)\widehat\varphi(\xi)$, thus $s_{\delta,\rho}$ is the Fourier multiplier of $-\mathscr L_\delta$. For fixed $\rho$ and $\delta$, the Riemann--Lebesgue lemma then gives
$s_{\delta,\rho}(\xi)\to2\mu_\rho\delta^{-2}$ as $|\xi|\to\infty$. Consequently, $s_{\delta,\rho}(\xi)^{-1}\to
\delta^2/(2\mu_\rho)=\beta\delta^2$, which determines the scaling of the mass inverse term in \eqref{2.18}.

\begin{lemma}\label{lemma3.1}
	Suppose that $k$ is even, nonnegative, and for some \(0<r\le1\),
	\begin{equation}\label{3.2}
		\|k\|_{L^1(B_1)}\le M_\rho,\qquad
		k\ge\rho_*>0\quad\text{a.e. in }B_r.
	\end{equation}
	Then there exists a constant $c_d>0$, depending only on $d$, such that
	\begin{equation}\label{3.3}
		2\rho_*c_dr^{d+2}
		\frac{|\eta|^2}{1+|\eta|^2}
		\le\Psi_k(\eta)
		\le8M_\rho\frac{|\eta|^2}{1+|\eta|^2}
		\quad\forall\eta\in\mathbb R^d.
	\end{equation}
	Consequently,
	\begin{equation}\label{3.4}
		2\rho_*c_dr^{d+2}
		\frac{|\xi|^2}{1+\delta^2|\xi|^2}
		\le s_{\delta,k}(\xi)
		\le8M_\rho\frac{|\xi|^2}{1+\delta^2|\xi|^2}
		\quad\forall\xi\in\mathbb R^d,
	\end{equation}
	where \(M_\rho\) and \(\rho_*\) are the kernel-class constants in \eqref{2.2}.
\end{lemma}
\begin{proof}
	The fact that $|\eta\cdot z|\le|\eta|$ for $z\in B_1$, together with $0\le1-\cos\alpha\le\min\{\alpha^2/2,2\}$, implies  $\Psi_k(\eta)\le\min\{M_\rho|\eta|^2,4M_\rho\}$ and hence the upper bound in \eqref{3.3}.
	
	For the lower bound, we define
	\begin{equation}\label{3.5}
		\Phi_d(s)=\int_{B_1}\bigl(1-\cos(sz_1)\bigr)\,dz.
	\end{equation}
	Set
	\begin{equation}\label{3.6}
		c_d:=\inf_{s>0}
		\frac{1+s^2}{s^2}\Phi_d(s).
	\end{equation}
	The function $(1+s^2)s^{-2}\Phi_d(s)$ is positive and continuous on $(0,\infty)$. Its limits as $s\to0$ and $s\to\infty$ are $\omega_d/[2(d+2)]$ and $\omega_d$, respectively. Hence $c_d>0$. Then the lower bound in \eqref{3.2}, the rotational symmetry of $B_r$, and the substitution $z=r\zeta$ yield
	\begin{equation*}
			\Psi_k(\eta)
			\ge2\rho_*r^d\Phi_d(r|\eta|)\ge2\rho_*c_dr^{d+2}
			\frac{|\eta|^2}{1+r^2|\eta|^2}\ge2\rho_*c_dr^{d+2}
			\frac{|\eta|^2}{1+|\eta|^2},
	\end{equation*}
	where the last inequality follows from $r\le1$. Combining the above estimate with the upper bound proves \eqref{3.3}. Moreover, setting $\eta=\delta\xi$ then proves \eqref{3.4} and ends the proof.
\end{proof}

Next, for $\varphi\in L^2(\mathbb R^d)$ and $u\in L^2(\Omega)$, we define
\begin{equation}\label{3.7}
	\begin{aligned}
		p_{t,\mathbb R^d}(\varphi)
		&=\inf_{w\in H^1(\mathbb R^d)}
		\left\{\|\nabla w\|_{L^2(\mathbb R^d)}^2
		+t^{-1}\|\varphi-w\|_{L^2(\mathbb R^d)}^2\right\},\\
		p_{t,\Omega}(u)
		&=\inf_{v\in H_0^1(\Omega)}
		\left\{\|\nabla v\|_{L^2(\Omega)}^2
		+t^{-1}\|u-v\|_{L^2(\Omega)}^2\right\}.
	\end{aligned}
\end{equation}
The following lemma gives a uniform comparison between $p_{t,\Omega}(u)$ and $p_{t,\mathbb R^d}(E_0u)$.

\begin{lemma}\label{lemma3.2}
	Let $\Omega$ be a bounded Lipschitz domain. Then there exist constants $\delta_0>0$ and $C_\Omega\ge1$ such that
	\begin{equation}\label{3.8}
		p_{t,\mathbb R^d}(E_0u)
		\le p_{t,\Omega}(u)
		\le C_\Omega p_{t,\mathbb R^d}(E_0u)
	\end{equation}
	for every $u\in L^2(\Omega)$, $0<\beta\le\beta_+$, and $0<\delta\le\delta_0$. Here $\delta_0$ depends only on $\Omega$ and $\beta_+$, and $C_\Omega$ depends only on $\Omega$.
\end{lemma}
\begin{proof}
	First, for any $v\in H_0^1(\Omega)$, extend $v$ by zero outside $\Omega$. Then $E_0v\in H^1(\mathbb R^d)$. Substituting $E_0v$ into the first minimization in \eqref{3.7} and taking the infimum over $v$ gives $p_{t,\mathbb R^d}(E_0u)\le p_{t,\Omega}(u)$.
	
	To prove the second inequality in \eqref{3.8}, set $\Omega_\varepsilon=\{x\in\Omega: \operatorname{dist}(x,\partial\Omega)<\varepsilon\}$. Cover $\partial\Omega$ by finitely many Lipschitz coordinate neighborhoods. After a translation and rotation in the $j$th neighborhood, the boundary is represented by $x_d=\psi_j(s)$ and $\Omega$ is locally given by $x_d>\psi_j(s)$, where $s\in\mathbb R^{d-1}$. Set $\tau=x_d-\psi_j(s)$. Since $\tau$ is comparable to the distance from the boundary in each Lipschitz coordinate neighborhood, there exist $L,\varepsilon_0>0$ such that every point in $\Omega_\varepsilon$, with $0<\varepsilon\le\varepsilon_0$, lies in one of these neighborhoods, satisfies $0<\tau<L\varepsilon$, and has its reflected point $(s,\psi_j(s)-\tau)$ in the same neighborhood. Since $\Omega$ is locally given by $x_d>\psi_j(s)$, the reflected point lies outside $\Omega$. The reflection allows us to compare the interior and exterior values of $w$ near the boundary.
	
	For smooth $w$, define $\widetilde w_j(s,\tau)=w(s,\psi_j(s)+\tau)$. Applying the fundamental theorem of calculus between $-\tau$ and $\tau$, followed by the Cauchy--Schwarz inequality, gives
	\begin{equation*}
		|\widetilde w_j(s,\tau)|^2
		\le2|\widetilde w_j(s,-\tau)|^2
		+4\tau\int_{-\tau}^{\tau}|\partial_\sigma\widetilde w_j(s,\sigma)|^2\,d\sigma,
		\qquad 0<\tau<L\varepsilon.
	\end{equation*}
	The change of variables $x=(s,\psi_j(s)+\tau)$ has unit Jacobian, and $|\partial_\tau\widetilde w_j|\le|\nabla w|$. Integrating the above local estimate over the finitely many coordinate neighborhoods and using their bounded overlap gives the following estimate for smooth $w$. Since smooth functions are dense in $H^1(\mathbb R^d)$, the same estimate holds for every $w\in H^1(\mathbb R^d)$:
	\begin{equation}\label{3.9}
		\|w\|_{L^2(\Omega_\varepsilon)}^2
		\le C_\Omega\left(
		\|w\|_{L^2(\mathbb R^d\setminus\Omega)}^2
		+\varepsilon^2\|\nabla w\|_{L^2(\mathbb R^d)}^2
		\right)
		\quad\forall w\in H^1(\mathbb R^d).
	\end{equation}
	Set $\delta_0=\varepsilon_0/\sqrt{\beta_+}$. For $0<\beta\le\beta_+$ and $0<\delta\le\delta_0$, we have $\sqrt t=\sqrt\beta\,\delta\le\varepsilon_0$. Take $\varepsilon=\sqrt t$, set $\chi_\varepsilon(x)=\min\{1,\operatorname{dist}(x,\partial\Omega)/\varepsilon\}$, and let $v=\chi_\varepsilon w|_\Omega\in H_0^1(\Omega)$. Since $\chi_\varepsilon=1$ on $\Omega\setminus\Omega_\varepsilon$ and $|\nabla\chi_\varepsilon|\le\varepsilon^{-1}$, the product rule and \eqref{3.9} give
	\begin{equation}\label{3.10}
		\|\nabla v\|_{L^2(\Omega)}^2
		+t^{-1}\|u-v\|_{L^2(\Omega)}^2
		\le C_\Omega\left(
		\|\nabla w\|_{L^2(\mathbb R^d)}^2
		+t^{-1}\|E_0u-w\|_{L^2(\mathbb R^d)}^2
		\right).
	\end{equation}
	Taking the infimum over $w$ in \eqref{3.10} proves the second inequality in \eqref{3.8} and ends the proof.
\end{proof}

We next establish a uniform energy equivalence on $\Omega$. Let $k\in L^1(B_1)$ be even and nonnegative, and define
\begin{equation}\label{3.11}
	\begin{aligned}
		a_{\delta,k}(u,v)
		&=\int_{\mathbb R^d}\int_{B_\delta(x)}
		\delta^{-(d+2)}k\!\left(\frac{y-x}{\delta}\right)
		\bigl(E_0u(y)-E_0u(x)\bigr)\\
		&\hspace{35mm}\times
		\bigl(E_0v(y)-E_0v(x)\bigr)\,dy\,dx .
	\end{aligned}
\end{equation}
The following theorem shows that $a_{\delta,k}(u,u)$ is uniformly equivalent to $p_{t,\Omega}(u)$.
\begin{theorem}\label{theorem3.3}
	Let $\Omega$ be a bounded Lipschitz domain and let $k$ satisfy the assumptions of Lemma \ref{lemma3.1}. Then for $\beta\in[\beta_-,\beta_+]$ and $0<\delta\le\delta_0$, where $\delta_0$ is given by Lemma \ref{lemma3.2}, we have
	\begin{equation}\label{3.12}
		c_0(r)p_{t,\Omega}(u)
		\le a_{\delta,k}(u,u)
		\le C_0p_{t,\Omega}(u)
		\quad\forall u\in L^2(\Omega),
	\end{equation}
	where
	\begin{equation}\label{3.13}
		c_0(r)=
		\frac{2\rho_*c_dr^{d+2}\min\{1,\beta_-\}}{C_\Omega},
		\qquad
		C_0=8M_\rho\max\{1,\beta_+\}.
	\end{equation}
\end{theorem}
\begin{proof}
	According to Plancherel's identity and Tonelli's theorem, for every $\varphi\in L^2(\mathbb R^d)$, we have
	\begin{equation}\label{3.14}
		\int_{\mathbb R^d}\int_{B_\delta(x)}
		\delta^{-(d+2)}k\!\left(\frac{y-x}{\delta}\right)
		|\varphi(y)-\varphi(x)|^2\,dy\,dx
		=\int_{\mathbb R^d}s_{\delta,k}(\xi)
		|\widehat\varphi(\xi)|^2\,d\xi.
	\end{equation}
	Moreover, applying Plancherel's identity to $\|\nabla w\|_{L^2(\mathbb R^d)}^2$ and
	$\|\varphi-w\|_{L^2(\mathbb R^d)}^2$ in \eqref{3.7}, and using $\widehat{\partial_jw}(\xi)=i\xi_j\widehat w(\xi)$, we obtain
	\begin{equation*}
		p_{t,\mathbb R^d}(\varphi)
		=\inf_{w\in H^1(\mathbb R^d)}
		\int_{\mathbb R^d}
		\left(
		|\xi|^2|\widehat w(\xi)|^2
		+t^{-1}|\widehat\varphi(\xi)-\widehat w(\xi)|^2
		\right)d\xi.
	\end{equation*}
	For each fixed $\xi$, minimizing
	$|\xi|^2|\widehat w(\xi)|^2
	+t^{-1}|\widehat\varphi(\xi)-\widehat w(\xi)|^2$
	with respect to $\widehat w(\xi)$ gives
	$\widehat w(\xi)=(1+t|\xi|^2)^{-1}\widehat\varphi(\xi)$, with
	minimum value
	$|\xi|^2(1+t|\xi|^2)^{-1}|\widehat\varphi(\xi)|^2$.
	By Plancherel's identity,
	\begin{equation*}
		\|w\|_{H^1(\mathbb R^d)}^2
		=\int_{\mathbb R^d}
		\frac{1+|\xi|^2}{(1+t|\xi|^2)^2}
		|\widehat\varphi(\xi)|^2\,d\xi
		\le C_t\|\varphi\|_{L^2(\mathbb R^d)}^2,
	\end{equation*}
	where $C_t>0$ depends only on $t$. Hence
	$w\in H^1(\mathbb R^d)$. Substituting the minimum value into the
	integral gives
	\begin{equation}\label{3.15}
		p_{t,\mathbb R^d}(\varphi)
		=\int_{\mathbb R^d}
		\frac{|\xi|^2}{1+t|\xi|^2}|\widehat\varphi(\xi)|^2\,d\xi.
	\end{equation}
	Using $t=\beta\delta^2$ and $\beta\in[\beta_-,\beta_+]$, we can obtain	
	\begin{equation}\label{3.16}
		\min\{1,\beta_-\}\frac{|\xi|^2}{1+t|\xi|^2}
		\le\frac{|\xi|^2}{1+\delta^2|\xi|^2}
		\le\max\{1,\beta_+\}\frac{|\xi|^2}{1+t|\xi|^2}.
	\end{equation}
	Then combining \eqref{3.4} in Lemma \ref{lemma3.1} with \eqref{3.14}, \eqref{3.15}, \eqref{3.16} and taking $\varphi=E_0u$ gives
	\begin{equation*}
		2\rho_*c_dr^{d+2}\min\{1,\beta_-\}
		p_{t,\mathbb R^d}(E_0u)
		\le a_{\delta,k}(u,u)
		\le8M_\rho\max\{1,\beta_+\}
		p_{t,\mathbb R^d}(E_0u).
	\end{equation*}
	
	Finally, Lemma \ref{lemma3.2} gives $C_\Omega^{-1}p_{t,\Omega}(u)\le p_{t,\mathbb R^d}(E_0u)\le p_{t,\Omega}(u)$. Substitution into the above bounds therefore proves \eqref{3.12} and \eqref{3.13}. 
\end{proof}

We now use Theorem \ref{theorem3.3} to prove that the variational problem \eqref{liftproblem} has a unique solution. By the Poincaré inequality, there exists a constant $C_P>0$, depending only on $\Omega$, such that $\|v\|_{L^2(\Omega)}\le C_P\|\nabla v\|_{L^2(\Omega)}$ for every $v\in H_0^1(\Omega)$. This inequality and \eqref{3.7} give $p_{t,\Omega}(u)\ge(C_P^2+t)^{-1}\|u\|_{L^2(\Omega)}^2$. Therefore, the lower bound in Theorem \ref{theorem3.3}, together with this estimate and $t\le\beta_+\delta_0^2$, gives

\begin{equation}\label{3.17}
	\begin{aligned}
		a_{\delta,k}(u,u)
		&\ge c_0(r)p_{t,\Omega}(u)\ge\frac{c_0(r)}{C_P^2+t}\|u\|_{L^2(\Omega)}^2\\
		&\ge\frac{c_0(r)}{C_P^2+\beta_+\delta_0^2}
		\|u\|_{L^2(\Omega)}^2
		\quad\forall u\in L^2(\Omega).
	\end{aligned}
\end{equation}

Taking $k=\rho$ and $r=r_*$ in \eqref{3.17} gives coercivity of $a_\delta$, and the bound $|a_\delta(u,v)|\le4\mu_\rho\delta^{-2}\|u\|_{L^2(\Omega)}\|v\|_{L^2(\Omega)}$ gives continuity. The Lax--Milgram lemma therefore gives a unique solution $w_\delta\in L^2(\Omega)$ to \eqref{liftproblem} for $0<\delta\le\delta_0$.
\section{Discrete approximation}\label{sec4}
We first establish the properties of the averaged kernel $\rho^{\rm av}$ defined in \eqref{2.8}, the exactness of the pair-cell assembly, and fixed-horizon convergence. We then derive the discrete energy equivalence and mass comparison used in Section \ref{sec5}.
\subsection{Pair-cell assembly and convergence}
\begin{proposition}\label{proposition4.1}
	Assume \eqref{2.1}, \eqref{2.2}, and the geometric and discretization assumptions of Section \ref{sec2.2} hold. Then the kernel $\rho^{\rm av}$ in \eqref{2.8}, extended by zero outside $P$, is even, nonnegative, and satisfies
	\begin{equation}\label{4.1}
		\|\rho^{\rm av}\|_{L^1(B_1)}\le M_\rho,\qquad
		\rho^{\rm av}(z)\ge\rho_*
		\quad\text{for a.e. }z\in B_{r_*/2}.
	\end{equation}
	Moreover, the matrix $A_h$ is symmetric positive definite for every $\delta>0$, and the pair-cell assembly in Section \ref{sec2.2} evaluates $A_h$ and the lifting contribution to $b_h$ exactly.
\end{proposition}
\begin{proof}
	By \eqref{2.8}, $\rho^{\rm av}$ is nonnegative. Since $\rho$ is even and the kernel partition is symmetric, $\mu_Z=\mu_{-Z}$ and $|Z|=|-Z|$, so $\rho^{\rm av}$ is even. Summing the cell masses gives
	\begin{equation*}
		\|\rho^{\rm av}\|_{L^1(B_1)}
		=\sum_{Z\in\mathscr Z_q}\mu_Z
		=\int_{P\cap B_1}\rho(z)\,dz
		\le M_\rho.
	\end{equation*}
	
	To prove the lower bound in \eqref{4.1}, take $z\in B_{r_*/2}$. Since $B_{r_*/2}\subset B_{1/2}\subset P$ and $\mathscr Z_q$ partitions $P$, there exists a cell $Z\in\mathscr Z_q$ containing $z$. Since $z\in B_{r_*/2}$ and $\operatorname{diam}Z\le q\le r_*/2$ by \eqref{2.7}, every $\zeta\in Z$ satisfies $|\zeta|\le|z|+\operatorname{diam}Z<r_*$. Thus $Z\subset B_{r_*}$, and \eqref{2.2} gives $\rho\ge\rho_*$ almost everywhere on $Z$. The cell average defined in \eqref{2.8} therefore satisfies $\rho^{\rm av}|_Z\ge\rho_*$, proving the lower bound in \eqref{4.1}.
	
	Moreover, for $u_h,v_h\in V_h^0$ with coefficient vectors $U,V$, the definition of $A_h$ gives
	\begin{equation}\label{4.2}
		U^TA_hV
		=a_{\delta,\rho^{\rm av}}(u_h,v_h).
	\end{equation}
	The bilinear form in \eqref{4.2} is symmetric. Taking $k=\rho^{\rm av}$ and $r=r_*/2$ in \eqref{3.4} gives $s_{\delta,\rho^{\rm av}}(\xi)>0$ for $\xi\ne0$. Then by \eqref{3.14}, $a_{\delta,\rho^{\rm av}}(u_h,u_h)>0$ whenever $u_h\ne0$. Equation \eqref{4.2} therefore shows that $A_h$ is symmetric positive definite for every $\delta>0$.
	
	Finally, on each pair cell $\mathcal C_{TSZ}$, $\rho^{\rm av}((y-x)/\delta)$ is constant. Since $E_0\phi_i$ and $\ell_{\delta,h}$ are affine on each mesh element, the differences $E_0\phi_i(y)-E_0\phi_i(x)$ and $\ell_{\delta,h}(y)-\ell_{\delta,h}(x)$ are affine in $(x,y)$ on this pair cell. Thus the integrands defining $A_h$ and the lifting term in $b_h$ are polynomials of total degree at most two. The quadrature rule \eqref{2.11} therefore evaluates their integrals exactly on each simplex. Thus the pair-cell assembly evaluates $A_h$ and the lifting term in $b_h$ exactly. 
\end{proof}
\begin{theorem}\label{theorem4.2}
	 Fix $0<\delta\le\delta_0$. Let $P_h$, $\mathscr Z_{q_h}$, and $\rho_h^{\rm av}$ denote the quantities $P$, $\mathscr Z_q$, and $\rho^{\rm av}$ from Section \ref{sec2.2} at mesh size $h$, and let $F_h\in L^2(\Omega)$. Assume that the hypotheses of Proposition \ref{proposition4.1} hold for each mesh in the family of Section \ref{sec2.2} and that
	 \begin{equation*}
	 	|B_1\setminus P_h|\to0,
	 	\qquad q_h\to0,
	 	\qquad \|F_h-F_\delta\|_{L^2(\Omega)}\to0
	 	\quad\text{as }h\to0.
	 \end{equation*}
	 Then the Galerkin solution $w_h\in V_h^0$ of \eqref{galerkin} converges in $L^2(\Omega)$ to the unique solution $w_\delta$ of \eqref{liftproblem}.
\end{theorem}
\begin{proof}
	We first show that $\varepsilon_h:=\|\rho_h^{\rm av}-\rho\|_{L^1(B_1)}\to0$. For any $\varphi\in C(\overline{B_1})$, define $\omega_\varphi(t)=\sup\{|\varphi(z)-\varphi(\zeta)|:z,\zeta\in\overline{B_1},\ |z-\zeta|\le t\}$ for $t\ge0$. The $L^1$ stability of cell averaging and the triangle inequality give
	\begin{equation*}
		\begin{aligned}
			\varepsilon_h
			&\le
			\int_{B_1\setminus P_h}|\rho(z)|\,dz
			+\sum_{Z\in\mathscr Z_{q_h}}
			\left(2\int_Z|\rho(z)-\varphi(z)|\,dz
			+|Z|\,\omega_\varphi(q_h)\right)\\
			&\le
			\int_{B_1\setminus P_h}|\rho(z)|\,dz
			+2\|\rho-\varphi\|_{L^1(B_1)}
			+|B_1|\,\omega_\varphi(q_h).
		\end{aligned}
	\end{equation*}
	The first term tends to zero since $\rho\in L^1(B_1)$ and $|B_1\setminus P_h|\to0$. Moreover, since continuous functions can approximate $\rho$ arbitrarily well in $L^1(B_1)$, we can choose $\varphi$ to make the second term arbitrarily small. With $\varphi$ fixed, the last term tends to zero by the uniform continuity of $\varphi$ and $q_h\to0$. Hence $\varepsilon_h\to0$.

	For $u,v\in L^2(\Omega)$, \eqref{3.11} and the Cauchy--Schwarz inequality give
	\begin{equation}\label{kernel-difference}
		|a_{\delta,\rho_h^{\rm av}}(u,v)-a_\delta(u,v)|
		\le4\delta^{-2}\varepsilon_h
		\|u\|_{L^2(\Omega)}\|v\|_{L^2(\Omega)}.
	\end{equation}
	For any $v_h\in V_h^0$, set $e_h=w_h-v_h$. The variational equations \eqref{galerkin} and \eqref{liftproblem} give
	\begin{equation*}
		\begin{aligned}
			a_{\delta,\rho_h^{\rm av}}(e_h,e_h)
			&=a_{\delta,\rho_h^{\rm av}}(w_\delta-v_h,e_h)
			+(F_h-F_\delta,e_h)_{L^2(\Omega)}\\
			&\quad+a_\delta(w_\delta,e_h)
			-a_{\delta,\rho_h^{\rm av}}(w_\delta,e_h).
		\end{aligned}
	\end{equation*}
	By Proposition \ref{proposition4.1}, $\rho_h^{\rm av}$ satisfies the kernel assumptions required for \eqref{3.17} with $r=r_*/2$. Applying \eqref{3.17} therefore gives a lower bound for the left-hand side. Combining this lower bound with continuity, the Cauchy--Schwarz inequality, and \eqref{kernel-difference} gives
	\begin{equation*}
		\begin{aligned}
			\frac{c_0(r_*/2)}{C_P^2+\beta_+\delta_0^2}
			\|e_h\|_{L^2(\Omega)}^2
			&\le a_{\delta,\rho_h^{\rm av}}(e_h,e_h)\le 4M_\rho\delta^{-2}
			\|w_\delta-v_h\|_{L^2(\Omega)}\|e_h\|_{L^2(\Omega)}\\
			&\quad+4\delta^{-2}\varepsilon_h
			\|w_\delta\|_{L^2(\Omega)}\|e_h\|_{L^2(\Omega)}
			+\|F_h-F_\delta\|_{L^2(\Omega)}\|e_h\|_{L^2(\Omega)}.
		\end{aligned}
	\end{equation*}
	The case $e_h=0$ is immediate. Otherwise, cancelling $\|e_h\|_{L^2(\Omega)}$ and applying the triangle inequality gives
	\begin{equation*}
		\begin{aligned}
			\|w_h-w_\delta\|_{L^2(\Omega)}
			&\le \|e_h\|_{L^2(\Omega)}
			+\|w_\delta-v_h\|_{L^2(\Omega)}\\
			&\le C_\delta\left(
			\|w_\delta-v_h\|_{L^2(\Omega)}
			+\varepsilon_h\|w_\delta\|_{L^2(\Omega)}
			+\|F_h-F_\delta\|_{L^2(\Omega)}
			\right).
		\end{aligned}
	\end{equation*}
	Taking the infimum over $v_h\in V_h^0$, we can obtain
	\begin{equation*}
		\|w_h-w_\delta\|_{L^2(\Omega)}
		\le C_\delta\left(
		\inf_{v_h\in V_h^0}\|w_\delta-v_h\|_{L^2(\Omega)}
		+\varepsilon_h\|w_\delta\|_{L^2(\Omega)}
		+\|F_h-F_\delta\|_{L^2(\Omega)}
		\right),
	\end{equation*}
	where $C_\delta$ is independent of $h$, $P_h$, and $q_h$. Since $C_c^\infty(\Omega)$ is dense in $L^2(\Omega)$, finite element interpolation gives $\inf_{v_h\in V_h^0}\|w_\delta-v_h\|_{L^2(\Omega)}\to0$ as $h\to0$. And the other terms tend to zero since $\varepsilon_h\to0$ and $\|F_h-F_\delta\|_{L^2(\Omega)}\to0$. Hence $w_h\to w_\delta$ in $L^2(\Omega)$. The proof is completed.
\end{proof}
\begin{remark}
	If $w_\delta\in H^2(\Omega)\cap H_0^1(\Omega)$, let $I_hw_\delta\in V_h^0$ be its Scott--Zhang interpolant. Summing the local estimates in \cite[(4.3)]{scott1990finite} gives $\|w_\delta-I_hw_\delta\|_{L^2(\Omega)}\le Ch^2|w_\delta|_{H^2(\Omega)}$. Using this interpolant in the best approximation term above gives
	\begin{equation*}
			\|w_h-w_\delta\|_{L^2(\Omega)}
			\le C_\delta\Bigl(
			h^2|w_\delta|_{H^2(\Omega)}
			+\varepsilon_h\|w_\delta\|_{L^2(\Omega)}+\|F_h-F_\delta\|_{L^2(\Omega)}
			\Bigr),
	\end{equation*}
	where $C_\delta$ is independent of $h$, $P_h$, and $q_h$. Consequently, $\|w_h-w_\delta\|_{L^2(\Omega)}=O(h^2)$ if $\varepsilon_h+\|F_h-F_\delta\|_{L^2(\Omega)}=O(h^2)$.
\end{remark}
\subsection{Discrete energy estimates}
To derive the discrete energy estimates, we use the $L^2$-orthogonal projection $\Pi_h:L^2(\Omega)\to V_h^0$. Under the mesh assumptions in Section \ref{sec2.2}, $\Pi_h$ is stable on $H_0^1(\Omega)$ \cite{bramble2002stability}. Together with orthogonality and the Poincaré inequality, we have
\begin{equation}\label{4.4}
	\|\Pi_hw\|_{L^2(\Omega)}^2\le\|w\|_{L^2(\Omega)}^2,\quad
	\|\nabla\Pi_hw\|_{L^2(\Omega)}^2
	\le C_\Pi\|\nabla w\|_{L^2(\Omega)}^2
\end{equation}
for every $w\in H_0^1(\Omega)$, where $C_\Pi\ge1$ is independent of $h$.

\begin{lemma}\label{lemma4.3}
	For $t>0$ and $u_h=\sum_iU_i\phi_i\in V_h^0$, define
	\begin{equation}\label{4.5}
		p_{t,h}(U)=
		\min_{W\in\mathbb R^{N_h}}
		\left\{
		W^TK_hW+t^{-1}(U-W)^TM_h(U-W)
		\right\}.
	\end{equation}
	Then we have
	\begin{equation}\label{4.6}
		p_{t,\Omega}(u_h)
		\le p_{t,h}(U)
		\le C_\Pi p_{t,\Omega}(u_h).
	\end{equation}
	Moreover, $p_{t,h}(U)=U^TH_hU$, where
	\begin{equation}\label{4.7}
		\begin{aligned}
			H_h
			&=M_h(M_h+tK_h)^{-1}K_h
			=K_h(M_h+tK_h)^{-1}M_h,\\
			H_h^{-1}
			&=K_h^{-1}+tM_h^{-1}.
		\end{aligned}
	\end{equation}
\end{lemma}
\begin{proof}
	We first compare $p_{t,\Omega}(u_h)$ and $p_{t,h}(U)$. By the definitions of $K_h$ and $M_h$, these quantities are obtained by minimizing the same energy over $H_0^1(\Omega)$ and $V_h^0$, respectively. Since $V_h^0\subset H_0^1(\Omega)$, the lower bound in \eqref{4.6} follows.
	
	To prove the upper bound, take any $w\in H_0^1(\Omega)$ and write $\Pi_hw=\sum_iW_i\phi_i$. Since $u_h\in V_h^0$, we have $\Pi_hu_h=u_h$ and hence $u_h-\Pi_hw=\Pi_h(u_h-w)$. Using $W$ in \eqref{4.5}, together with \eqref{4.4} and $C_\Pi\ge1$, we have
	\begin{equation*}
		\begin{aligned}
			p_{t,h}(U)
			&\le \|\nabla\Pi_hw\|_{L^2(\Omega)}^2
			+t^{-1}\|u_h-\Pi_hw\|_{L^2(\Omega)}^2\\
			&\le C_\Pi\left(
			\|\nabla w\|_{L^2(\Omega)}^2
			+t^{-1}\|u_h-w\|_{L^2(\Omega)}^2
			\right).
		\end{aligned}
	\end{equation*}
	Taking the infimum over $w\in H_0^1(\Omega)$ proves the upper bound in \eqref{4.6}.
	
	Next, to derive the expression for $H_h$ in \eqref{4.7}, we differentiate the expression $W^TK_hW+t^{-1}(U-W)^TM_h(U-W)$ in \eqref{4.5} with respect to $W$. Setting the derivative to zero gives $(M_h+tK_h)W_*=M_hU$ for the minimizer $W_*$. Substituting $W_*=(M_h+tK_h)^{-1}M_hU$ into \eqref{4.5} gives $p_{t,h}(U)=U^TH_hU$, where
	\begin{equation*}
			H_h=t^{-1}\left[M_h-M_h(M_h+tK_h)^{-1}M_h\right]=M_h(M_h+tK_h)^{-1}K_h.
	\end{equation*}
	Its inverse is
	\begin{equation*}
		H_h^{-1}
		=K_h^{-1}(M_h+tK_h)M_h^{-1}
		=K_h^{-1}+tM_h^{-1}.
	\end{equation*}
	This inverse is symmetric positive definite, so $H_h$ is also symmetric positive definite. Transposing its expression above gives $H_h=K_h(M_h+tK_h)^{-1}M_h$, which proves \eqref{4.7} and ends the proof. 
\end{proof}

With \(t=\beta\delta^2\), we refer to \(H_h^{-1}=K_h^{-1}+\beta\delta^2M_h^{-1}\) in \eqref{4.7} as the ideal additive preconditioner. The practical preconditioner \eqref{2.18} replaces \(K_h^{-1}\) and \(M_h^{-1}\) by \(G_h\) and \(D_h^{-1}\), respectively. Assumption \ref{assumption2.1} gives the required comparison between \(G_h\) and \(K_h^{-1}\). We next establish the corresponding comparison between \(D_h^{-1}\) and \(M_h^{-1}\). For each $T\in\mathcal T_h$, let $M_T$ be the element mass matrix obtained by integrating products of the local basis functions, and define $D_T$ as the diagonal matrix whose diagonal entries are the row sums of $M_T$. Using the simplex integration formulas to evaluate $M_T$ and then taking its row sums, we can obtain
\begin{equation*}
	M_T=\frac{|T|}{(d+1)(d+2)}
	\left(I_{d+1}+\mathbf1\mathbf1^T\right),
	\qquad
	D_T=\frac{|T|}{d+1}I_{d+1},
\end{equation*}
where $I_{d+1}$ is the identity and $\mathbf1=(1,\ldots,1)^T$.

To compare $M_T$ and $D_T$, we write $X\preceq Y$ for symmetric matrices $X$ and $Y$ if $Y-X$ is positive semidefinite. Since the eigenvalues of $I_{d+1}+\mathbf1\mathbf1^T$ are $1$ and $d+2$, the formulas above give
\begin{equation*}
	\frac1{d+2}D_T\preceq M_T\preceq D_T.
\end{equation*}
Summing over the elements and restricting to the free nodes, we can obtain
\begin{equation}\label{4.8}
	\frac1{d+2}D_h\preceq M_h\preceq D_h,\quad
	\frac1{d+2}M_h^{-1}\preceq D_h^{-1}\preceq M_h^{-1}.
\end{equation}
\section{Proof of Theorem \ref{theorem2.2}}\label{sec5}
Let $\delta_0$ be given by Lemma \ref{lemma3.2} and fix $0<\delta\le\delta_0$. Set
\begin{equation}
	c_B=\min\{c_G,1/(d+2)\},\quad
	C_B=\max\{C_G,1\},\quad
	c=\frac{c_0(r_*/2)c_B}{C_\Pi},\quad
	C=C_0C_B.
\end{equation}
For $u_h\in V_h^0$ with nodal coefficient vector $U$, Proposition \ref{proposition4.1} and Theorem \ref{theorem3.3} with $k=\rho^{\rm av}$ and $r=r_*/2$ give
\begin{equation}\label{5.1}
	c_0(r_*/2)p_{t,\Omega}(u_h)
	\le U^TA_hU
	\le C_0p_{t,\Omega}(u_h).
\end{equation}
Combining \eqref{5.1} with \eqref{4.6} in Lemma \ref{lemma4.3} yields
\begin{equation}\label{5.2}
	\frac{c_0(r_*/2)}{C_\Pi}H_h
	\preceq A_h
	\preceq C_0H_h.
\end{equation}
For the preconditioner $B_h=G_h+tD_h^{-1}$ defined in \eqref{2.18}, Assumption \ref{assumption2.1} and \eqref{4.8} give
\begin{equation*}
	c_GK_h^{-1}+\frac{t}{d+2}M_h^{-1}
	\preceq B_h
	\preceq C_GK_h^{-1}+tM_h^{-1}.
\end{equation*}
Using the definitions of $c_B$ and $C_B$ and the identity $H_h^{-1}=K_h^{-1}+tM_h^{-1}$ in \eqref{4.7}, we can obtain
\begin{equation}\label{5.3}
	c_BH_h^{-1}
	\preceq B_h
	\preceq C_BH_h^{-1}.
\end{equation}
Consequently,
\begin{equation}\label{5.4}
	C_B^{-1}H_h
	\preceq B_h^{-1}
	\preceq c_B^{-1}H_h.
\end{equation}
By \eqref{5.4}, $c_BB_h^{-1}\preceq H_h\preceq C_BB_h^{-1}$. Substituting these bounds into \eqref{5.2} gives
\begin{equation*}
	cB_h^{-1}
	\preceq \frac{c_0(r_*/2)}{C_\Pi}H_h
	\preceq A_h
	\preceq C_0H_h
	\preceq CB_h^{-1}.
\end{equation*}
These matrix inequalities imply that, for any nonzero $X\in\mathbb R^{N_h}$,
\begin{equation}\label{5.5}
	c\le
	\frac{X^TA_hX}
	{X^TB_h^{-1}X}
	\le C.
\end{equation}
Setting $X=B_h^{1/2}Y$ in \eqref{5.5} gives \eqref{2.19} and hence \eqref{2.20}. The proof is completed.
\qed

\section{Numerical experiments}\label{sec6}
In this section, we give several numerical experiments on $\Omega=(0,1)^2$ to test the proposed method. In section \ref{sec6.1}, we examine the convergence of the finite element method for nonlocal diffusion problem, and the remaining experiments study the preconditioner and its local solver. All the computations are performed in MATLAB R2025b on a MacBook Pro laptop with an Apple M5 Max processor and 128 GB of memory.
\subsection{Finite element convergence}\label{sec6.1}
We test finite element convergence for nonlocal diffusion \eqref{modelequation} with the proposed pair-cell assembly. We set $g_\delta=0$, $\delta=0.2$, and use the kernel
\begin{equation}\label{6.1}
	\rho(z)=\frac{4}{\pi},
	\quad z\in B_1.
\end{equation}
For this kernel, one can verify that $\mu_\rho=4$ and $\beta=1/8$. Moreover, we choose the exact solution
\begin{equation}\label{6.2}
	u_\delta(x_1,x_2)
	=
	\begin{cases}
		x_1(1-x_1)x_2(1-x_2),&(x_1,x_2)\in\Omega,\\
		0,&(x_1,x_2)\in\Omega_I^\delta,
	\end{cases}.
\end{equation}
The corresponding source $f_\delta=-\mathscr L_\delta u_\delta$ is evaluated by a 20-point Gauss quadrature rule. We approximate this problem by \eqref{galerkin} on uniform triangular meshes with $n=4,8,16,32,64$ subdivisions in each coordinate direction. Here $h=\sqrt2/n$, and $w_h$ denotes the approximation of $u_\delta|_\Omega$. For each mesh, we approximate the reference ball $B_1$ by a regular inscribed polygon $P_h\subset B_1$ with $m=4n$ sides. Since the kernel is constant, its cell average is exact on $P_h$, and
\begin{equation*}
	\|\rho_h^{\rm av}-\rho\|_{L^1(B_1)}
	=\frac4\pi |B_1\setminus P_h|
	=\frac4\pi\left(\pi-\frac m2\sin\frac{2\pi}{m}\right)
	=O(h^2).
\end{equation*}
A symmetric kernel partition is refined so that $q_h\to0$. Since $\rho$ is constant, this refinement does not change
$\rho_h^{\rm av}$. Moreover, $F_h=F_\delta=f_\delta$ and $u_\delta|_\Omega\in H^2(\Omega)\cap H_0^1(\Omega)$. Therefore, the estimate above and the remark following Theorem~\ref{theorem4.2} give $\|w_h-u_\delta\|_{L^2(\Omega)}=O(h^2)$.

We assemble $A_h$ by the pair-cell method with the quadrature rule \eqref{2.11} and solve $A_hU=b_h$ by Cholesky factorization, and the relative $L^2$ error is
\begin{equation}\label{6.3}
	E_{L^2,h}
	=\frac{\|u_\delta-w_h\|_{L^2(\Omega)}}{\|u_\delta\|_{L^2(\Omega)}}.
\end{equation}
The load integrals $(f_\delta,\phi_i)_{L^2(\Omega)}$ and the $L^2$ norms in \eqref{6.3} are evaluated by Gauss rules with 24 and 6 points, respectively. Table \ref{tab:fem-convergence} shows second-order convergence, in agreement with the estimate following Theorem \ref{theorem4.2}.
\begin{table}[htbp]
	\centering
	\caption{Relative $L^2$ errors and convergence orders.}
	\label{tab:fem-convergence}
	\begin{tabular}{@{}rrrrr@{}}
		\toprule
		\multicolumn{1}{c}{$N_h$}& \multicolumn{1}{c}{$m$}& \multicolumn{1}{c}{$h$}& \multicolumn{1}{c}{$E_{L^2,h}$}& \multicolumn{1}{c}{order} \\
		\midrule
		9 &  16 & $3.5355\times10^{-1}$ & $8.2123\times10^{-2}$ & -- \\
		49 &  32 & $1.7678\times10^{-1}$ & $1.8223\times10^{-2}$ & 2.17 \\
		225 &  64 & $8.8388\times10^{-2}$ & $4.4836\times10^{-3}$ & 2.02 \\
		961 & 128 & $4.4194\times10^{-2}$ & $1.1269\times10^{-3}$ & 1.99 \\
		3969 & 256 & $2.2097\times10^{-2}$ & $2.8361\times10^{-4}$ & 1.99 \\
		\bottomrule
	\end{tabular}
\end{table}

\subsection{Mesh and horizon robustness}\label{sec6.2}
This experiment tests the robustness of $B_h$ with respect to the mesh size and horizon. We take $g_\delta$ and $\rho$ as in Section \ref{sec6.1}. The reference ball is approximated by the regular inscribed polygon $P_{16}$ with a symmetric 64-cell partition and $q=1/2$. Starting from a mesh with 34 triangles, we apply five successive uniform refinements, each subdividing every triangle into four, and use the resulting meshes with $N_h=53,241,1025,4225,17153$. On level $j=0,\ldots,4$, we take $\delta/h=0.8$, $0.8\,2^j$, and $0.8\,2^{-j}$ for the constant, increasing, and decreasing paths, respectively. The paths share their first parameter pair and contain 13 distinct pairs in total.

Moreover, we choose $G_h$ in \eqref{2.18} as a symmetric geometric multigrid V-cycle with three weighted Jacobi pre- and post-smoothing steps. On level $\ell$, the weight is $\omega_\ell=1.7/\lambda_{\max}(\operatorname{diag}(K_\ell)^{-1}K_\ell)$.

For each parameter pair, we apply PCG with zero initial guess and relative residual tolerance $10^{-8}$ to smooth, random, and high-frequency loads. For each run, we recompute
\begin{equation}\label{6.4}
	r_{\rm true}=\frac{\|b_h-A_hU_{\rm out}\|_2}{\|b_h\|_2}
\end{equation}
and report the maximum iteration count and $r_{\rm true}$ over the three runs. The extremal eigenvalues of $B_h^{1/2}A_hB_h^{1/2}$ are computed directly for $N_h\le4225$. And for $N_h = 17153$, the Arnoldi iteration is applied to the similar matrix $B_hA_h$ to obtain the extremal eigenvalues. Figure \ref{fig:mesh-horizon-robustness} shows the extremal eigenvalues of $B_h^{1/2}A_hB_h^{1/2}$ along the three paths, and Table \ref{tab:mesh-horizon-robustness} reports the corresponding condition numbers and PCG results. It can be found that along all three paths, the condition number is at most $5.840$, and PCG requires at most 22 iterations. These results confirm the robustness of $B_h$ with respect to both $h$ and $\delta$. 

\begin{figure}[htbp]
	\centering
	\subfigure[Constant $\delta/h$]{%
		\includegraphics[width=0.31\textwidth]{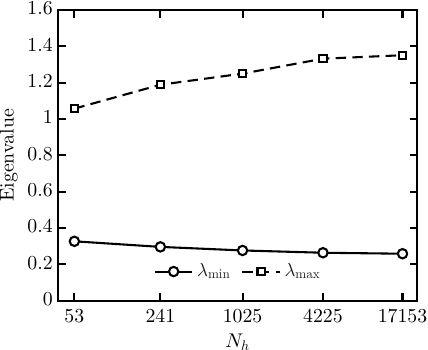}%
		\label{fig:mesh-horizon-robustness-a}}
	\hfill
	\subfigure[Increasing $\delta/h$]{%
		\includegraphics[width=0.31\textwidth]{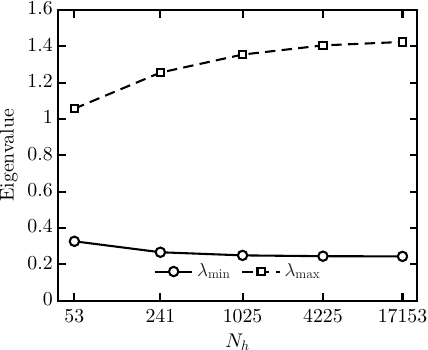}%
		\label{fig:mesh-horizon-robustness-b}}
	\hfill
	\subfigure[Decreasing $\delta/h$]{%
		\includegraphics[width=0.31\textwidth]{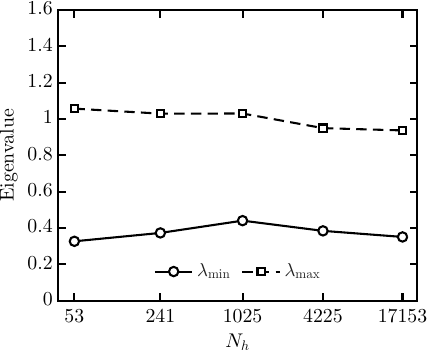}%
		\label{fig:mesh-horizon-robustness-c}}
	\caption{Extremal eigenvalues for the three $\delta/h$ paths.}
	\label{fig:mesh-horizon-robustness}
\end{figure}
\begin{table}[htbp]
	\centering
	\caption{Condition numbers and PCG results for three $\delta/h$ paths.}
	\label{tab:mesh-horizon-robustness}
	\small
	\setlength{\tabcolsep}{3.5pt}
	\begin{tabular}{@{}lrrrrrr@{}}
		\toprule
		\multicolumn{1}{c}{path}
		& \multicolumn{1}{c}{$N_h$}
		& \multicolumn{1}{c}{$h$}
		& \multicolumn{1}{c}{$\delta/h$}
		& \multicolumn{1}{c}{$\kappa$}
		& \multicolumn{1}{c}{\shortstack{PCG\\iterations}}
		& \multicolumn{1}{c}{\shortstack{max.\\$r_{\rm true}$}} \\
		\midrule
		constant
		&   53 & 0.305604 & 0.8 & 3.233 & 15
		& $5.82\times10^{-9}$ \\
		constant
		&  241 & 0.152802 & 0.8 & 4.012 & 17
		& $8.97\times10^{-9}$ \\
		constant
		& 1025 & 0.076401 & 0.8 & 4.519 & 19
		& $4.45\times10^{-9}$ \\
		constant
		& 4225 & 0.038200 & 0.8 & 5.045 & 20
		& $8.76\times10^{-9}$ \\
		constant
		& 17153 & 0.019100 & 0.8 & 5.224 & 20
		& $9.21\times10^{-9}$ \\
		\addlinespace
		increasing
		&   53 & 0.305604 & 0.8 & 3.233 & 15
		& $5.82\times10^{-9}$ \\
		increasing
		&  241 & 0.152802 & 1.6 & 4.708 & 19
		& $9.91\times10^{-9}$ \\
		increasing
		& 1025 & 0.076401 & 3.2 & 5.430 & 21
		& $7.31\times10^{-9}$ \\
		increasing
		& 4225 & 0.038200 & 6.4 & 5.735 & 22
		& $6.39\times10^{-9}$ \\
		increasing
		& 17153 & 0.019100 & 12.8 & 5.840 & 22
		& $7.15\times10^{-9}$ \\
		\addlinespace
		decreasing
		&   53 & 0.305604 & 0.8 & 3.233 & 15
		& $5.82\times10^{-9}$ \\
		decreasing
		&  241 & 0.152802 & 0.4 & 2.760 & 14
		& $7.28\times10^{-9}$ \\
		decreasing
		& 1025 & 0.076401 & 0.2 & 2.338 & 12
		& $5.21\times10^{-9}$ \\
		decreasing
		& 4225 & 0.038200 & 0.1 & 2.471 & 13
		& $6.07\times10^{-9}$ \\
		decreasing
		& 17153 & 0.019100 & 0.05 & 2.670 & 14
		& $6.16\times10^{-9}$ \\
		\bottomrule
	\end{tabular}
\end{table}

Besides, to examine the role of the mass correction, we compare the local Poisson preconditioner $K_h^{-1}$, the ideal additive preconditioner $H_h^{-1}=K_h^{-1}+\beta\delta^2M_h^{-1}$ in \eqref{4.7}, and the practical preconditioner $B_h$ in \eqref{2.18} for $N_h=53$ and $\delta/h=0.8$. Table \ref{tab:preconditioner-comparison} shows that the mass correction reduces the condition number from $30.344$ to $1.882$. Moreover, when the exact inverses are replaced by $G_h$ and $D_h^{-1}$, the practical preconditioner $B_h$ gives a condition number of $3.233$, which remains small.
\begin{table}[htbp]
	\centering
	\caption{Preconditioner comparison at $N_h=53$ and $\delta/h=0.8$.}
	\label{tab:preconditioner-comparison}
	\begin{tabular}{@{}lrr@{}}
		\toprule
		\multicolumn{1}{c}{preconditioner}
		& \multicolumn{1}{c}{condition number}
		& \multicolumn{1}{c}{PCG iterations} \\
		\midrule
		$K_h^{-1}$                        & 30.344 & 28--35 \\
		$K_h^{-1}+\beta\delta^2M_h^{-1}$ &  1.882 & 10 \\
		$B_h$                             &  3.233 & 15 \\
		\bottomrule
	\end{tabular}
\end{table}
\subsection{Kernel dependence}
To test the robustness of $B_h$ with respect to kernel shape, we consider the constant kernel $\rho_0$, the discontinuous angular step kernel $\rho_{\rm step}$, the smooth kernel $\rho_{\rm mix}$ combining second- and fourth-order angular modes, the smooth fourfold kernel $\rho_4$, and the narrow double-cone kernel $\rho_{\rm cone}$. For $z=r(\cos\theta,\sin\theta)\in B_1\setminus\{0\}$, set $c_{\rm ref}=4/\pi$, $R=19$, and $\phi=\omega=\pi/16$. Let $e_\phi=(\cos\phi,\sin\phi)$ and $e_\phi^\perp=(-\sin\phi,\cos\phi)$, and define
\begin{equation*}
	C_{\phi,\omega}
	=\left\{\zeta\in B_1:
	|e_\phi^\perp\cdot\zeta|
	\le \tan(\omega)|e_\phi\cdot\zeta|\right\}.
\end{equation*}
Then the five kernels are defined by
\begin{equation}\label{6.5}
	\begin{aligned}
		\rho_0(z)
		&=c_{\rm ref},\\
		\rho_{\rm step}(z)
		&=c_{\rm ref}\left[1+0.9\operatorname{sgn}(\cos2\theta)\right],\\
		\rho_{\rm mix}(z)
		&=c_{\rm ref}\left[
		1+0.45r^2\cos2\left(\theta-\frac{\pi}{16}\right)
		+0.45r^4\cos4\left(\theta-\frac{\pi}{32}\right)
		\right],\\
		\rho_4(z)
		&=c_{\rm ref}\left[
		1+0.9r^4\cos4\left(\theta-\frac{\pi}{32}\right)
		\right],\\
		\rho_{\rm cone}(z)
		&=\frac{4}{\pi+2\omega(R-1)}
		\left[1+(R-1)\mathbf1_{C_{\phi,\omega}}(z)\right],
	\end{aligned}
\end{equation}
where $\mathbf1_E$ denotes the indicator of $E$. One can verify that for all five kernels, $\beta=1/8$ and \eqref{2.1}--\eqref{2.2} hold with the common constants $M_\rho=4$, $r_*=1$, and $\rho_*=0.4/\pi$. The five kernels
are shown in Figure \ref{fig:reference-kernels}.

\begin{figure}[htbp]
	\centering
	\subfigure[$\rho_0$]{%
		\includegraphics[width=0.31\textwidth,trim=2mm 2mm 2mm 2mm,clip]{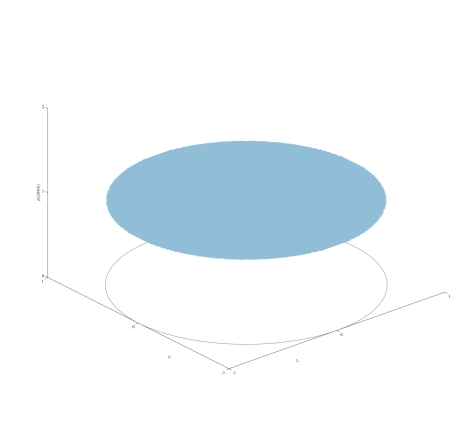}%
		\label{fig:reference-kernels-a}}
	\hfill
	\subfigure[$\rho_{\rm step}$]{%
		\includegraphics[width=0.31\textwidth,trim=2mm 2mm 2mm 2mm,clip]{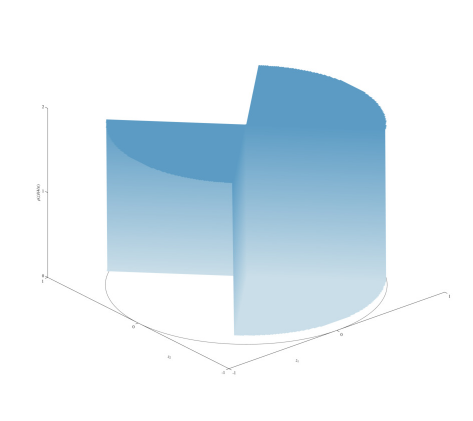}%
		\label{fig:reference-kernels-b}}
	\hfill
	\subfigure[$\rho_{\rm mix}$]{%
		\includegraphics[width=0.31\textwidth,trim=2mm 2mm 2mm 2mm,clip]{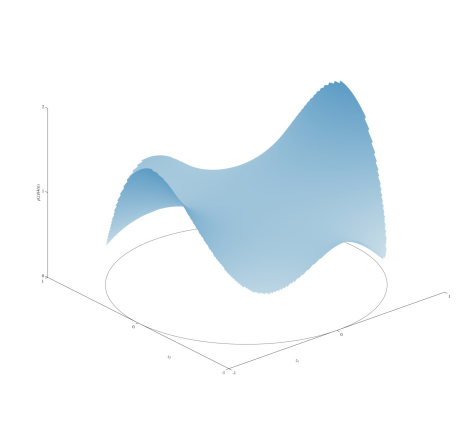}%
		\label{fig:reference-kernels-c}}
	\par\vspace{-1mm}
	\subfigure[$\rho_4$]{%
		\includegraphics[width=0.31\textwidth,trim=2mm 2mm 2mm 2mm,clip]{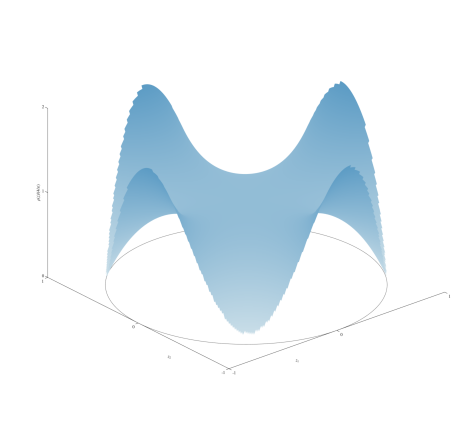}%
		\label{fig:reference-kernels-d}}
	\hspace{0.04\textwidth}
	\subfigure[$\rho_{\rm cone}$]{%
		\includegraphics[width=0.31\textwidth,trim=2mm 2mm 2mm 2mm,clip]{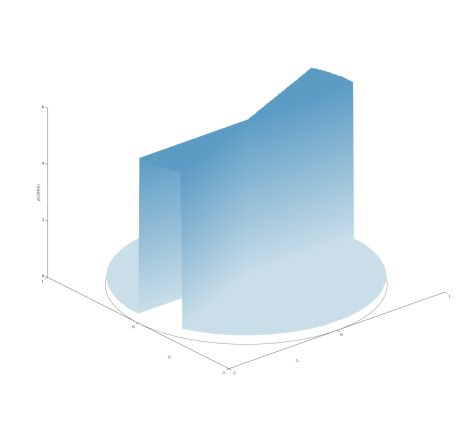}%
		\label{fig:reference-kernels-e}}
	\caption{Plots of the five kernel functions.}
	\label{fig:reference-kernels}
\end{figure}

Table \ref{tab:kernel-robustness} reports the results using the settings of Section \ref{sec6.2} and a symmetric 64-cell partition of $P_{16}$ with $q=1/2$. Across all 35 tests, $\kappa<7.45$ and PCG requires at most 25 iterations. Thus $B_h$ remains effective for the tested kernel class.

\begin{table}[htbp]
	\centering
	\caption{Condition numbers; maximum PCG iterations are in parentheses.}
	\label{tab:kernel-robustness}
	\small
	\setlength{\tabcolsep}{3.5pt}
	\begin{tabular}{@{}rrrrrrr@{}}
		\toprule
		\multicolumn{1}{c}{$N_h$}
		& \multicolumn{1}{c}{$\delta/h$}
		& \multicolumn{1}{c}{$\rho_0$}
		& \multicolumn{1}{c}{$\rho_{\rm step}$}
		& \multicolumn{1}{c}{$\rho_{\rm mix}$}
		& \multicolumn{1}{c}{$\rho_4$}
		& \multicolumn{1}{c}{$\rho_{\rm cone}$} \\
		\midrule
		53 & 0.8 & $3.233\,(15)$ & $4.088\,(17)$ & $3.356\,(15)$
		& $3.238\,(15)$ & $4.637\,(18)$ \\
		241 & 0.4 & $2.760\,(14)$ & $4.128\,(18)$ & $2.952\,(14)$
		& $2.804\,(14)$ & $5.618\,(21)$ \\
		241 & 0.8 & $4.012\,(17)$ & $5.521\,(20)$ & $4.252\,(18)$
		& $4.027\,(17)$ & $6.259\,(22)$ \\
		241 & 1.6 & $4.708\,(19)$ & $6.098\,(22)$ & $4.863\,(20)$
		& $4.715\,(19)$ & $6.534\,(23)$ \\
		1025 & 0.2 & $2.338\,(12)$ & $5.549\,(21)$ & $2.576\,(13)$
		& $2.342\,(12)$ & $6.910\,(24)$ \\
		1025 & 0.8 & $4.519\,(19)$ & $6.422\,(23)$ & $4.814\,(19)$
		& $4.620\,(19)$ & $7.406\,(25)$ \\
		1025 & 3.2 & $5.430\,(21)$ & $7.174\,(24)$ & $5.633\,(21)$
		& $5.523\,(21)$ & $7.443\,(25)$ \\
		\bottomrule
	\end{tabular}
\end{table}

\section{Conclusions}\label{sec7}
In this paper, we develop a preconditioner $B_h=G_h+\beta\delta^2D_h^{-1}$ for finite element discretizations of volume-constrained nonlocal diffusion problems that is robust with respect to both the mesh size $h$ and the horizon $\delta$. The preconditioner combines a local Poisson solver with a lumped-mass correction, while the proposed pair-cell assembly exactly realizes the averaged-kernel Galerkin matrix and preserves its symmetry and positive definiteness. We prove that under the stated assumptions, the preconditioned spectrum is bounded independently of $h$, $\delta$, and the kernel partition. The numerical results confirm the predicted $L^2$ convergence of the finite element solution, and the robustness of the preconditioner across the tested meshes, horizons, and kernels.


\section*{Funding}
J. Lu's work is partially supported by the National Natural Science Foundation of China under Grant number 12501565. Y. Nie's work is partially supported by the National Natural Science Foundation of China under Grant number 12571440. 

\bibliographystyle{unsrt} 
\bibliography{mybibfile}

@article{braess1983new,
	title={A new convergence proof for the multigrid method including the V-cycle},
	author={Braess, Dietrich and Hackbusch, Wolfgang},
	journal={SIAM Journal on Numerical Analysis},
	volume={20},
	number={5},
	pages={967--975},
	year={1983},
	publisher={SIAM}
}

@article{brannick2008uniform,
	title={Uniform convergence of the multigrid V-cycle on graded meshes for corner singularities},
	author={Brannick, James J and Li, Hengguang and Zikatanov, Ludmil T},
	journal={Numerical Linear Algebra with Applications},
	volume={15},
	number={2-3},
	pages={291--306},
	year={2008},
	publisher={Wiley Online Library}
}

@book{saad2003iterative,
	title={Iterative methods for sparse linear systems},
	author={Saad, Yousef},
	year={2003},
	publisher={SIAM}
}

@article{scott1990finite,
	title={Finite element interpolation of nonsmooth functions satisfying boundary conditions},
	author={Scott, L Ridgway and Zhang, Shangyou},
	journal={Mathematics of Computation},
	volume={54},
	number={190},
	pages={483--493},
	year={1990}
}

@article{bramble2002stability,
	title={On the stability of the {$L^2$} projection in {$H^1$ ($\Omega$)}},
	author={Bramble, James and Pasciak, Joseph and Steinbach, Olaf},
	journal={Mathematics of Computation},
	volume={71},
	number={237},
	pages={147--156},
	year={2002}
}

@article{silling2000reformulation,
	title={Reformulation of elasticity theory for discontinuities and long-range forces},
	author={Silling, Stewart A},
	journal={Journal of the Mechanics and Physics of Solids},
	volume={48},
	number={1},
	pages={175--209},
	year={2000},
	publisher={Elsevier}
}

@article{dipasquale2014crack,
	title={Crack propagation with adaptive grid refinement in 2D peridynamics},
	author={Dipasquale, Daniele and Zaccariotto, Mirco and Galvanetto, Ugo},
	journal={International Journal of Fracture},
	volume={190},
	number={1},
	pages={1--22},
	year={2014},
	publisher={Springer}
}

@article{panchadhara2016application,
	title={Application of peridynamic stress intensity factors to dynamic fracture initiation and propagation},
	author={Panchadhara, Rohan and Gordon, Peter A},
	journal={International Journal of Fracture},
	volume={201},
	number={1},
	pages={81--96},
	year={2016},
	publisher={Springer}
}

@article{zhou2025research,
	title={Research on the crack propagation and bifurcation under dynamic indentation with peridynamics},
	author={Zhou, Ping and Duan, Wenbo and Peng, Kai and Guo, Dongming},
	journal={Engineering Fracture Mechanics},
	volume={315},
	pages={110804},
	year={2025},
	publisher={Elsevier}
}

@article{du2012analysis,
	title={Analysis and approximation of nonlocal diffusion problems with volume constraints},
	author={Du, Qiang and Gunzburger, Max and Lehoucq, Richard B and Zhou, Kun},
	journal={SIAM Review},
	volume={54},
	number={4},
	pages={667--696},
	year={2012},
	publisher={SIAM}
}

@article{du2014nonlocal,
	title={Nonlocal  convection-diffusion volume-constrained problems and jump processes},
	author={Du, Qiang and Huang, Zhan and Lehoucq, Richard B},
	journal={Discrete \& Continuous Dynamical Systems-Series B},
	volume={19},
	number={4},
	pages={961},
	year={2014}
}

@article{d2013fractional,
	title={The fractional Laplacian operator on bounded domains as a special case of the nonlocal diffusion operator},
	author={D'Elia, Marta and Gunzburger, Max},
	journal={Computers \& Mathematics with Applications},
	volume={66},
	number={7},
	pages={1245--1260},
	year={2013},
	publisher={Elsevier}
}

@article{tian2016asymptotically,
	title={Asymptotically compatible schemes for the approximation of fractional Laplacian and related nonlocal diffusion problems on bounded domains},
	author={Tian, Xiaochuan and Du, Qiang and Gunzburger, Max},
	journal={Advances in Computational Mathematics},
	volume={42},
	number={6},
	pages={1363--1380},
	year={2016},
	publisher={Springer}
}

@incollection{du2023nonlocal,
	title={Nonlocal diffusion models with consistent local and fractional limits},
	author={Du, Qiang and Tian, Xiaochuan and Zhou, Zhi},
	booktitle={A$^3$N$^2$M: Approximation, Applications, and Analysis of Nonlocal, Nonlinear Models: Proceedings of the 50th John H. Barrett Memorial Lectures},
	pages={175--213},
	year={2023},
	publisher={Springer}
}

@article{du2025error,
	title={Error estimates of finite element methods for nonlocal problems using exact or approximated interaction neighbourhoods},
	author={Du, Qiang and Xie, Hehu and Yin, Xiaobo and Zhang, Jiwei},
	journal={IMA Journal of Numerical Analysis},
	pages={draf106},
	year={2025},
	publisher={Oxford University Press}
}

@article{tian2014asymptotically,
	title={Asymptotically compatible schemes and applications to robust discretization of nonlocal models},
	author={Tian, Xiaochuan and Du, Qiang},
	journal={SIAM Journal on Numerical Analysis},
	volume={52},
	number={4},
	pages={1641--1665},
	year={2014},
	publisher={SIAM}
}

@article{d2021cookbook,
	title={A cookbook for approximating Euclidean balls and for quadrature rules in finite element methods for nonlocal problems},
	author={D'Elia, Marta and Gunzburger, Max and Vollmann, Christian},
	journal={Mathematical Models and Methods in Applied Sciences},
	volume={31},
	number={08},
	pages={1505--1567},
	year={2021},
	publisher={World Scientific}
}

@article{chen2024efficient,
	title={{Efficient implementation of 3D FEM for nonlocal Poisson problem with different ball approximation strategies}},
	author={Chen, Gengjian and Ma, Yuheng and Zhang, Jiwei},
	journal={Communications in Computational Physics},
	volume={36},
	number={5},
	pages={1378--1410},
	year={2024}
}

@article{lu2026nonlocal,
	title={A nonlocal nonlinear {Schr{\"o}dinger} model: well-posedness, local limit, and structure-preservingasymptotically compatible Fourier approximations},
	author={Lu, Jiashu and Nie, Yufeng and Xie, Xinning and Zhang, Pingrui},
	journal={arXiv preprint arXiv:2608.09372},
	year={2026}
}

@article{chen2017convergence,
	title={Convergence analysis of a multigrid method for a nonlocal model},
	author={Chen, Minghua and Deng, Weihua},
	journal={SIAM Journal on Matrix Analysis and Applications},
	volume={38},
	number={3},
	pages={869--890},
	year={2017},
	publisher={SIAM}
}

@article{chen2024fast,
	title={Fast algebraic multigrid for block-structured dense systems arising from nonlocal diffusion problems},
	author={Chen, Minghua and Cao, Rongjun and Serra-Capizzano, Stefano},
	journal={Calcolo},
	volume={61},
	number={4},
	pages={57},
	year={2024},
	publisher={Springer}
}

@article{liu2025fast,
	title={A fast computational framework for the linear peridynamic model},
	author={Liu, Chenguang and Tian, Hao and Don, Wai Sun and Wang, Hong},
	journal={Engineering with Computers},
	volume={41},
	number={2},
	pages={965--988},
	year={2025},
	publisher={Springer}
}

@article{tian2024fast,
	title={A fast implementation of the linear bond-based peridynamic beam model},
	author={Tian, Hao and Yang, Xianchu and Liu, Chenguang and Liu, Guilin},
	journal={Advances in Applied Mathematics and Mechanics},
	volume={16},
	number={2},
	pages={305--330},
	year={2024}
}

@article{schuster2025schwarz,
	title={Schwarz Methods for nonlocal problems},
	author={Schuster, Matthias and Vollmann, Christian and Schulz, Volker},
	journal={Journal of Peridynamics and Nonlocal Modeling},
	volume={7},
	number={3},
	pages={8},
	year={2025},
	publisher={Springer}
}

@article{zhang2026spectral,
	title={Spectral Analysis of $\tau$-Preconditioners for Two-Level Toeplitz Systems with Applications in Nonlocal Diffusion Models},
	author={Zhang, Jiali and Sun, Tao and Sun, Haiwei and Zhang, Jiwei},
	journal={Journal of Scientific Computing},
	volume={106},
	number={1},
	pages={21},
	year={2026},
	publisher={Springer}
}

\end{document}